\documentclass[12pt,a4paper]{article}

\usepackage{amssymb, amsmath, amsthm}

\usepackage[cm]{fullpage}
\usepackage[english]{babel}
\usepackage[cp1250]{inputenc}
\usepackage[T1]{fontenc}
\usepackage[pdftex]{graphicx}
\usepackage{booktabs}
\usepackage{multirow}

\newtheorem{thm}{Theorem}
\newtheorem{lem}{Lemma}
\newtheorem{prop}{Proposition}

\newtheorem{cor}{Corollary} 

\title{Numerical method for the time-fractional porous medium equation}
\author{\L ukasz P\l ociniczak\thanks{Faculty of Pure and Applied Mathematics, Wroc{\l}aw University of Science and Technology, Wyb. Wyspia{\'n}skiego 27, 50-370 Wroc{\l}aw, Poland}$\;^,$\footnote{Email: lukasz.plociniczak@pwr.edu.pl}}
\date{}

\begin{document}
\maketitle
	
\begin{abstract}	
	This papers deals with a construction and convergence analysis of a finite difference scheme for solving time-fractional porous medium equation. The governing equation exhibits both nonlocal and nonlinear behaviour making the numerical computations challenging. Our strategy is to reduce the problem into a single one-dimensional Volterra integral equation for the self-similar solution and then to apply the discretization. The main difficulty arises due to the non-Lipschitzian behaviour of the equation's nonlinearity. By the analysis of the recurrence relation for the error we are able to prove that there exists a family of finite difference methods that is convergent for a large subset of the parameter space. We illustrate our results with a concrete example of a method based on the midpoint quadrature. \\
		
\end{abstract}

\section{Introduction}

In our previous investigations \cite{Plo13,Plo14,Plo15} we have considered the following time-fractional nonlinear diffusion equation (also known as a time-fractional porous medium equation)
\begin{equation}
\partial^\alpha_t u = \left(u^m u_x\right)_x, \quad 0<\alpha< 1, \quad m>1,
\label{eqn:DiffEqPDE}
\end{equation}
where the fractional derivative is of the Riemann-Liouville type
\begin{equation}
\partial^\alpha_t u(x,t) = \frac{1}{\Gamma(1-\alpha)}\frac{\partial}{\partial t}\int_0^t (t-s)^{-\alpha} u(x,s) ds.
\end{equation}
The boundary conditions that we impose are the following
\begin{equation}
u(x,0) = 0, \quad u(0,t) = 1, \quad x>0, \quad t>0,
\label{eqn:CondPDE}
\end{equation}
where the nondimensionalization has been implicitly assumed. The above problem is a description of an experiment where one measures the material properties of an essentailly one-dimensional half-infinite medium under the water imbibition from the boundary (for a experimental results see \cite{El04,Ram08,Kun01,De06,Zho17}). 

There is a wealth of literature concerning numerical methods for fractional differential equations (both ODEs and PDEs) - some surveys can be found in \cite{Die02,Bal12}. Concerning the linear fractional diffusion the Reader can find relevant methods for example in \cite{Tad06,Li09,Yus05}. The numerical methods for the main operator in the spatial diffusion, namely the fractional Laplacian, can be found in \cite{Hua04,Cus18}. As for the nonlinear version of the anomalous diffusion, an interesting paper recently appeared  in which the authors constructed a multi-grid waveform method of fast convergence \cite{Gas17}. Moreover, a similar equation to ours has been solved in \cite{Awo16} in the context of petroleum industry. Some other numerical approaches concern space-fractional nonlinear diffusion \cite{Erv07,Del14}, nonlinear source terms \cite{Zhu09,Zen14,Bhr16} and variable order diffusion \cite{Mog17}. 

Our approach is based on a transformation of the governing PDE to the equivalent nonlinear Volterra equation (for a recent survey of theory and numerical methods see \cite{Bru17}). In that case the classical convergence theory cannot be applied since the nonlinearity of the equation is non-Lipschitzian. According to our best knowledge, there is a scarce literature concerning similar problems. A very interesting paper is \cite{Buc97} where an iterative technique has been applied and convergence proofs given. Moreover, in \cite{Cor00} a short summary of the theoretical and numerical character has been published. Lastly, we mention our own work \cite{Plo17} from which the present considerations stem where we have given the convergence proof assuming the kernel's separation from zero. In the present discussion we relax this assumption. 

The paper is structured as follows. In the second section we formulate the problem in terms of the Volterra setting starting from the self-similar form of the time-fractional porous medium equation. We proceed to the main results in the third section where a construction of a convergent finite difference method is given. We end the paper with some numerical simulations illustrating our results. 

\section{Problem statement}

The straightforward numerical approach is to start with the equation (\ref{eqn:DiffEqPDE}). This has been done for example in \cite{Plo14} (but also see the recent approach in \cite{Gas17}) but Authors noted a very large computational cost and stability issues caused by the interplay of two factors: nonlocality (of the fractional derivarive) and nonlinearity (of the flux). A more sensible method is to transform (\ref{eqn:DiffEqPDE}) into its self-similar form and then to derive an appropriate numerical scheme. We begin by necessary preparations. 

If we substitute $\eta = x t^{-\alpha/2}$ for some $0<\alpha< 1$ and denote $u(x,t)=U(x t^{-\alpha/2})$ we arrive at
\begin{equation}
\left(U^m U'\right)' = \left[(1-\alpha)-\frac{\alpha}{2}\eta \frac{d}{d\eta}\right]I^{0, 1-\alpha}_{-\frac{2}{\alpha}} U, \quad 0<\alpha< 1,
\label{eqn:mainEq}
\end{equation} 
with the boundary conditions
\begin{equation}
U(0) = 1, \quad U(\infty) = 0.
\label{eqn:mainEqBC}
\end{equation}
Here, $I^{a,b}_c$ is the Erd\'elyi-Kober operator \cite{Kir93,Kir97,Sne75}
\begin{equation}
I^{a,b}_c U(\eta) = \frac{1}{\Gamma(b)}\int_0^1 (1-s)^{b-1}s^a U(s^\frac{1}{c}\eta)ds.
\label{eqn:EK}
\end{equation}
Notice that (\ref{eqn:mainEq}) is an ordinary integro-differential equation which should be more tractable for numerical work than (\ref{eqn:DiffEqPDE}). 

There is a one more transformation that can be done in order to simplify the matters even more. Notice that due to (\ref{eqn:mainEqBC}) our problem has a free-boundary which can cause difficulties to resolve numerically (for some details see \cite{Cra87}). However, there exists a substitution that can take (\ref{eqn:mainEq}) into an equivalent initial-value problem. Physical situation as well as theoretical considerations \cite{Atk71,Plo17a,Plo18} suggest that the solution of (\ref{eqn:mainEq}) has a compact support, i.e. there exists a point $\eta^* \geq 0$ such that
\begin{equation}
	U(\eta) = 0 \quad \text{for} \quad \eta\geq \eta^*.
	\label{eqn:eta*}
\end{equation}
This simply means that the wetting front propagates at a finite speed. Now, we can substitute
\begin{equation}
U(\eta)=\left(m (\eta^*)^2 \right)^\frac{1}{m} y(z), \quad z=1-\frac{\eta}{\eta^*},
\label{eqn:transformation}
\end{equation}
which changes (\ref{eqn:mainEq}) into
\begin{equation}
m(y^m y')' =\left[(1-\alpha)+\frac{\alpha}{2}(1-z)\frac{d}{dz}\right]F_{\alpha}y, \quad 0<\alpha<1, \quad 0\leq z\leq 1.
\label{eqn:mainEq2}
\end{equation}
where the linear operator $F_{\alpha}$ is defined by
\begin{equation}
F_{\alpha}y(z):= \frac{1}{\Gamma(1-\alpha)}\int_{(1-z)^\frac{\alpha}{2}}^{1} (1-s)^{-\alpha} y(1-s^{-\frac{2}{\alpha}}(1-z)) ds.
\label{eqn:F}
\end{equation}
Further, the same transformation yields the formula for wetting front position
\begin{equation}
\eta^* = \frac{1}{\sqrt{m y(1)^m}}.
\end{equation}
Lastly, equation (\ref{eqn:mainEq2}) can be integrated twice and transformed into a Volterra integral equation
\begin{equation}
y(z)^{m+1} = \frac{m+1}{m}\int_0^z \left(\frac{\alpha}{2}+\left(1-\frac{\alpha}{2}\right)z-t\right) F_\alpha y(t)dt,
\label{eqn:FixedPoint}
\end{equation}
which is of the main interest for this paper. It is a Volterra equation in which the nonlinear term in non-Lipschitz (it can be seen by introducing a new function $u=y^{1+m}$).

We can manipulate the integrand of (\ref{eqn:FixedPoint}) in order to explicitly write the kernel. This would simplify the subsequent numerical implementation. To start, change the variable $u=1-s^{-\frac{2}{\alpha}}(1-z)$ in the definition of $F_\alpha$ given by (\ref{eqn:F})
\begin{equation}
	F_\alpha y(t)  = \frac{\alpha}{2} \frac{(1-t)^\frac{\alpha}{2}}{\Gamma(1-\alpha)}\int_0^t \left(1-\left(\frac{1-z}{1-u}\right)^\frac{\alpha}{2}\right)^{-\alpha} \left(1-u\right)^{-\frac{\alpha}{2}-1}y(u)du.
\end{equation}
Now, plugging the above formula into (\ref{eqn:FixedPoint}) and changing the order of integration we obtain
\begin{equation}
	y(z)^{m+1} = \frac{m+1}{m}\frac{\alpha}{2} \frac{1}{\Gamma(1-\alpha)} \int_0^z \left[\int_u^z\left(\frac{\alpha}{2}+\left(1-\frac{\alpha}{2}\right)z-t\right)(1-t)^\frac{\alpha}{2}\left(1-\left(\frac{1-t}{1-u}\right)^\frac{\alpha}{2}\right)^{-\alpha}dt\right]\frac{y(u)}{(1-u)^{\frac{\alpha}{2}+1}}du.
\end{equation}
To simplify further we can substitute $s = \left(\frac{1-t}{1-u}\right)^\frac{\alpha}{2}$ and arrive at
\begin{equation}
y(z)^{m+1} = \int_0^z \left[\underbrace{\frac{m+1}{m}\frac{1}{\Gamma(1-\alpha)}\int_{\left(\frac{1-z}{1-u}\right)^\frac{\alpha}{2}}^1\left(s^\frac{2}{\alpha}(1-u)-\left(1-\frac{\alpha}{2}\right)(1-z)\right)s^\frac{2}{\alpha}\left(1-s\right)^{-\alpha}}_{K(z,u)}\right]y(u)du.
\label{eqn:kernel}
\end{equation}
Finally, we can notice the kernel $K$ can be written in terms of the Incomplete Beta Function defined by
\begin{equation}
	B(x,a,b) := \int_0^x (1-t)^{a-1} t^{b-1} dt.
\end{equation}
Hence,
\begin{equation}
	y(z)^{m+1} = \int_0^z K(z,u) y(u) du,
\label{eqn:mainEqInt}
\end{equation}
where for $0\leq u\leq z$ we have
\begin{equation}
\begin{split}
	K(z,u) &:= \frac{m+1}{m}\frac{1}{\Gamma(1-\alpha)} \left[(1-u)\left(B\left(\frac{4}{\alpha}+1,1-\alpha\right)-B\left(\left(\frac{1-z}{1-u}\right)^\frac{\alpha}{2},\frac{4}{\alpha}+1,1-\alpha\right)\right)
	\right. \\
	&\left. -\left(1-\frac{\alpha}{2}\right)(1-z)\left(B\left(\frac{2}{\alpha}+1,1-\alpha\right)-B\left(\left(\frac{1-z}{1-u}\right)^\frac{\alpha}{2},\frac{2}{\alpha}+1,1-\alpha\right)\right)
	\right].
\end{split}
\label{eqn:kernelExp}
\end{equation}
It is evident that the kernel is positive and continuous with a singular derivative. A class of equations similar to the above has been a subject of very active investigations. For important results stating the existence and uniqueness the Reader is invited to consult \cite{Bus76,Gri81,Okr89}.

\section{Finite difference scheme}

Let us begin with introducing the grid
\begin{equation}
	z_n := \frac{n}{N}, \quad h:=\frac{1}{N}, \quad n=0,1,...,N.
\end{equation}
Following our previous investigations \cite{Plo17} we discretize the integral (\ref{eqn:mainEqInt})
\begin{equation}
\int_0^{z_n} K(z_n,t) y(t) dt = h \sum_{i=1}^{n-1} w_{n,i} K(z_n,z_i) y(z_i) + \delta_n(h),
\label{eqn:LocCons}
\end{equation}
where $\delta_n(h)$ is the local consistency error. Furthermore, define the maximal error
\begin{equation}
\delta(h) := \max_{1\leq n\leq N} |\delta_n(h)|.
\label{eqn:delta}
\end{equation}
For the weights, we assume their boundedness
\begin{equation}
	0<w_{n,i} \leq W.
\label{eqn:weights}
\end{equation}
We can propose the following finite difference scheme for solving (\ref{eqn:mainEqInt})
\begin{equation}
	y_n^{m+1} = h \sum_{i=1}^{n-1} w_{n,i} K_{n,i} \; y_i, \quad n=2,3,...,N,
\label{eqn:NumMet}
\end{equation} 
where $K_{n,i} := K(z_n,z_i)$. 

The objective of the following is to prove that (\ref{eqn:NumMet}) is convergent provided we choose an appropriate starting value $y_1$. From (\ref{eqn:kernelExp}) we notice that $K(1,1) = 0$ and thus the kernel is not separated from zero. Hence, we cannot use our previous results from \cite{Plo17}. However, several ingredients from our developing theory will be useful. The first is a simple consequence of the integral equation structure and discretization scheme.
\begin{prop}
	Let $y$ be the nontrivial solution of (\ref{eqn:mainEqInt}) and $y_n$ are constructed via the iteration (\ref{eqn:NumMet}). If $\delta_n(h)$ in (\ref{eqn:LocCons}) is non-negative (non-positive) for all $n=1,2,...,N$, then $y(n h) \geq y_n$ ($y(nh) \leq y_n$) provided that $y(h) \geq y_1$ ($y(h)\leq y_1$).
	\label{prop:Mono}
\end{prop}
This monotonicity result has been proved in \cite{Plo17}. The second piece of information about the solution of (\ref{eqn:mainEqInt}) concerns the \emph{global} estimates (for the proofs of the next two lemmas see \cite{Plo17a}).
\begin{lem}
\label{lem:estimates}
	Let $y$ be the nontrivial solution of (\ref{eqn:mainEqInt}). Then, the following estimates take place
	\begin{equation}
		C_1 z^\frac{2-\alpha}{m} \leq y(z) \leq C_2 z^\frac{2-\alpha}{m}, \quad 0\leq z\leq 1, \quad 0<\alpha\leq 1,
	\end{equation}
	where $C_{1,2}$ are given by
	\begin{equation}
	\begin{array}{ll}
	C_1 = \left\{
	\begin{array}{ll}
	\left(\left(\frac{\alpha}{2}\right)^{1-\alpha} \frac{\Gamma\left(\frac{2-\alpha}{m}\right)}{\Gamma\left(2-\alpha+\frac{2-\alpha}{m}\right)} \frac{1}{2-\alpha+m(3-\alpha)}\right)^\frac{1}{m+1}, & 0<\alpha\leq 1-\frac{1}{m+1}; \\
	
	\left(\left(\frac{\alpha}{2}\right)^{2-\alpha} \frac{\Gamma\left(1+\frac{2-\alpha}{m}\right)}{\Gamma\left(2-\alpha+\frac{2-\alpha}{m}\right)} \frac{1}{2-\alpha}\right)^\frac{1}{m+1}, & 1-\frac{1}{m+1} < \alpha \leq 1,
	\end{array}
	\right.
	\vspace{12pt}\\ 
	C_2 = \Gamma(3-\alpha)^{-\frac{1}{m+1}}.
	\end{array}
	\label{eqn:Gammas}
	\end{equation}
\end{lem}
Our solution is thus bounded from below and above by a power function of the same class. As it appears, something more detailed can be said about the behaviour of $y$ at the origin. 
\begin{lem}
	\label{thm:IC}
	The solution $y=y(z)$ of (\ref{eqn:mainEqInt}) satisfies
	\begin{equation}
	y(z) \sim \left(\frac{\alpha}{2}\right)^{2-\alpha} \frac{\Gamma\left(\frac{2-\alpha}{m}\right)}{\Gamma\left(1-\alpha+\frac{2-\alpha}{m}\right)} \frac{z^\frac{2-\alpha}{m}}{(2-\alpha)\left(1+\frac{1}{m}\right)-1} \quad \text{as} \quad z\rightarrow 0^+, \quad 0<\alpha\leq 1.
	\label{eqn:Asym}
	\end{equation}
\end{lem}
The asymptotics is thus of power-type and the constant of proportionality is known. As a quick application of the above lemma we propose a sensible choice of the starting value of the numerical method (\ref{eqn:NumMet})
\begin{equation}
	y_1 := \left(\frac{\alpha}{2}\right)^{2-\alpha} \frac{\Gamma\left(\frac{2-\alpha}{m}\right)}{\Gamma\left(1-\alpha+\frac{2-\alpha}{m}\right)} \frac{h^\frac{2-\alpha}{m}}{(2-\alpha)\left(1+\frac{1}{m}\right)-1},
\label{eqn:startingValue}
\end{equation}
which introduces the starting error of higher order than $h^\frac{2-\alpha}{m}$ as $h\rightarrow 0^+$. 

Now, we proceed to the proof of the fact that (\ref{eqn:NumMet}) is convergent. First, we state an auxiliary lemma which can be thought as a generalization of the discreet Gronwall-Bellman's Lemma (for a thorough review of similar results see \cite{Ame97}). 
\begin{lem}
\label{lem:recur}
	Let $\left\{e_n\right\}$, $n=1,2,...$ be a sequence of positive numbers satisfying 
	\begin{equation}
	e_n \leq \frac{1}{n^\beta} \left(A \sum_{i=1}^{n-1} (n-i)^\gamma \; e_i + B\right), \quad n\geq 2,
	\label{eqn:LemRecur}
	\end{equation}
	where $A$, $B$ are positive constants, $\beta \geq 1$ and $\gamma \geq 0$. Then, provided that $e_1 \leq B$ we have
	\begin{equation}
		e_n \leq B\;\frac{f(n)}{n^\beta},
	\label{eqn:LemRecurResult}
	\end{equation}
	where $f(1) = 1$ and 
	\begin{equation}
		f(n) = 1 + A n^\gamma \left[\prod_{j=2}^{n-1} \left(1+A j^{\gamma-\beta}\right) + \sum_{i=2}^{n-1} \frac{1}{i^\beta} \prod_{j=i+1}^{n-1} \left(1+A j^{\gamma-\beta}\right)\right], \quad n\geq 2.
	\label{eqn:RecurEqSol}
	\end{equation}
\end{lem}
\begin{proof}
The proof proceeds by mathematical induction. The right-hand side of (\ref{eqn:LemRecurResult}) reduces to $B$ for $n=1$ (with the usual convention for the product: $\prod_{i=1}^{0}=1$) which yields $e_1 \leq B$. The initial inductive step is thus satisfied by the assumption. 

Assume now that (\ref{eqn:LemRecurResult}) is satisfied for $(n-1)$-th term. We will show that this also is the case for $e_{n}$. To this end, use the inductive assumption to obtain
\begin{equation}
	e_{n} \leq \frac{B}{n^\beta} \left(A \sum_{i=1}^{n-1} (n-i)^\gamma \; \frac{f(i)}{i^\beta} + 1\right).
\end{equation}
We can immediately estimate the sum to obtain
\begin{equation}
e_{n} \leq \frac{B}{n^\beta} \left(A n^\gamma\sum_{i=1}^{n-1} \frac{f(i)}{i^\beta} + 1\right).
\end{equation}
In order to make the above inequality to satisfy the assertion we require that
\begin{equation}
	A n^\gamma \sum_{i=1}^{n-1} \frac{f(i)}{i^\beta} + 1 = f(n), \quad f(1)  = 1.
\label{eqn:RecurEq}
\end{equation}
We will show that the solution of this nonlocal recurrence equation is equal to (\ref{eqn:RecurEqSol}). Define $g(n) := \sum_{i=1}^{n-1} f(i) i^{-\beta}$ and notice that
\begin{equation}
	g(n+1)-g(n) = \frac{f(n)}{n^\beta}, \quad g(2) = 1.
\end{equation}
Hence, thanks to (\ref{eqn:RecurEq}) we have
\begin{equation}
	A n ^\gamma g(n)+1 = n^\beta \left(g(n+1)-g(n)\right), \quad g(2) = 1.
\end{equation}
After rearranging we obtain the following nonhomogeneous recurrence relation
\begin{equation}
	g(n+1) = n^{-\beta} + \left(1+A n^{\gamma-\beta}\right)g(n), \quad g(2) = 1.
\end{equation}
To simplify the notation we temporarily introduce $a(n):=n^{-\beta}$ and $b(n):=1+An^{\gamma-\beta}$ and solve the following equation by the successive iteration
\begin{equation}
\begin{split}
	g(n+1) &= a(n) + b(n) g(n) = a(n) + a(n-1) b(n) + b(n) b(n-1) g(n-1) = \\
	&= a(n) + a(n-1) b(n) + a(n-2) b(n-1)b(n) + b(n)b(n-1)b(n-2)g(n-2) = ...
\end{split}
\end{equation} 
Continuing in this inductive fashion we can show that
\begin{equation}
	g(n+1) = \prod_{j=2}^{n} b(j) + \sum_{i=2}^{n} a(i) \prod_{j=i+1}^{n} b(j),
\end{equation}
where the convention $\prod_{j=2}^{1}=1$ is used. Going back to the original variables we have
\begin{equation}
	g(n+1) = \prod_{j=2}^{n} \left(1+A j^{\gamma-\beta}\right) + \sum_{i=2}^{n} \frac{1}{i^\beta} \prod_{j=i+1}^{n} \left(1+A j^{\gamma-\beta}\right), \quad n\geq 1,
\end{equation}
and if we use (\ref{eqn:RecurEq}) and the definition of $g(n)$ we can state the result in terms of the $f(n)$ function
\begin{equation}
	f(n) = 1 + A n^\gamma \left[\prod_{j=2}^{n-1} \left(1+A j^{\gamma-\beta}\right) + \sum_{i=2}^{n-1} \frac{1}{i^\beta} \prod_{j=i+1}^{n-1} \left(1+A j^{\gamma-\beta}\right)\right], \quad n\geq 2.
\end{equation}
This concludes the proof.
\end{proof}

Having the above in hand we can state the main result.
\begin{thm}
\label{thm:convergence}
Assume that $0\leq z\leq X<1$ and let $y_n$ be the calculated from (\ref{eqn:NumMet}) approximation of the solution of (\ref{eqn:mainEqInt}). Moreover, define the error $e_n:=y_n-y(nh)$ and assume that $|e_1|\leq \frac{1}{(m+1)C_1^m}\frac{\delta(h)}{h^{2-\alpha}}$, where $C_1$ is from Lemma \ref{lem:estimates}. Then, for a quadrature with $\delta_n(h) \leq 0$ for every $n\in\mathbb{N}$ we have
\begin{equation}
	|e_n| \leq const. \times \delta(h)h^{\alpha-A-1} \quad \text{as} \quad h\rightarrow 0^+ \quad \text{with} \quad nh\rightarrow z_n,
\end{equation}
with $A= \frac{W D}{(m+1)C_1^m}$, where $W$ is from (\ref{eqn:weights}) and $D=\frac{m+1}{2m}\frac{1}{\Gamma(2-\alpha)} \left(\frac{1}{1-X}\right)^{1-\alpha}$.
\end{thm}
\begin{proof}
Notice that by the assumption and Proposition \ref{prop:Mono} we have $e_n:=y_n-y(nh)\geq 0$. Using the Lagrange's Mean-Value Theorem we obtain
\begin{equation}
	y_n^{m+1}-y(nh)^{m+1} = (m+1)\xi_n^m e_n,
\end{equation}
where $y(nh) \leq \xi_n \leq y_n$. On the other hand, from (\ref{eqn:mainEqInt}) and (\ref{eqn:LocCons}) 
\begin{equation}
	y_n^{m+1}-y(nh)^{m+1} = h\sum_{i=1}^{n-1} w_{n,i}K_{n,i} e_i - \delta_n(h) \leq h\sum_{i=1}^{n-1} w_{n,i}K_{n,i} e_i + \delta(h).
\end{equation}
Therefore, from (\ref{eqn:weights})
\begin{equation}
	(m+1) y(nh)^m e_n \leq W h\sum_{i=1}^{n-1} K_{n,i} e_i + \delta(h).
\end{equation}
Now, using Lemma \ref{lem:estimates} we can estimate the left-hand side
\begin{equation}
	(m+1) C_1^m (nh)^{2-\alpha} e_n \leq W h\sum_{i=1}^{n-1} K_{n,i} e_i + \delta(h).
\label{eqn:theoremEst1}
\end{equation}
The next step is to find an appropriate bound for the kernel $K$. To this end we go back to (\ref{eqn:kernel}) and notice that
\begin{equation}
	s^\frac{\alpha}{2}(1-u)-\left(1-\frac{\alpha}{2}\right)(1-z) = s^\frac{\alpha}{2}(1-u) \left[1-s^{-\frac{\alpha}{2}}\left(1-\frac{\alpha}{2}\right)\frac{1-z}{1-u} \right] \leq \frac{1}{2},
\end{equation}
because $(\frac{1-z}{1-u})^{\alpha/2}\leq s \leq 1$. Hence,
\begin{equation}
	K(z,u) \leq \frac{m+1}{2m}\frac{1}{\Gamma(1-\alpha)} \int_{(\frac{1-z}{1-u})^{\alpha/2}}^{1}(1-s)^{-\alpha}ds = \frac{m+1}{2m}\frac{1}{\Gamma(2-\alpha)} \left(1-\left(\frac{1-z}{1-u}\right)^{\alpha/2}\right)^{1-\alpha}.
\end{equation}
Moreover, by convexity
\begin{equation}
	\left(1-\left(\frac{1-z}{1-u}\right)^{\alpha/2}\right)^{1-\alpha} = \left(1-\left(1-\frac{z-u}{1-u}\right)^{\alpha/2}\right)^{1-\alpha} \leq \left(\frac{z-u}{1-u}\right)^{1-\alpha}.
\end{equation}
Finally,
\begin{equation}
	K(z,u) \leq \frac{m+1}{2m}\frac{1}{\Gamma(2-\alpha)} \left(\frac{z-u}{1-X}\right)^{1-\alpha} =: D \left(z-u\right)^{1-\alpha},
\end{equation}
where we have used the assumption that $0\leq z\leq X<1$. 

Now, going back to (\ref{eqn:theoremEst1}) and noticing that $z_n-z_i = h(n-i)$ we can write
\begin{equation}
\begin{split}
	e_n &\leq \frac{1}{n^{2-\alpha}}\left(\frac{W D}{(m+1)C_1^m}\sum_{i=1}^{n-1}(n-i)^{1-\alpha} e_i + \frac{1}{(m+1)C_1^m}\frac{\delta(h)}{h^{2-\alpha}}\right) \\ 
	&=: \frac{1}{n^{2-\alpha}}\left(A\sum_{i=1}^{n-1}(n-i)^{1-\alpha} e_i +B\right).
\end{split}
\end{equation}
This form of the inequality can be plugged into Lemma \ref{lem:recur} to yield 
\begin{equation}
	e_n \leq \frac{1}{(m+1)C_1^m}\frac{\delta(h)}{(nh)^{2-\alpha}} f(n).
\end{equation}
It is interesting for us to learn how does the function $f$ behave asymptotically as $n\rightarrow \infty$. To see the exact order we write it explicitly as
\begin{equation}
	f(n) = 1 + A n^{1-\alpha} \left[\prod_{j=2}^{n-1} \left(1+\frac{A}{j}\right) + \sum_{i=2}^{n-1} \frac{1}{i^{2-\alpha}} \prod_{j=i+1}^{n-1} \left(1+\frac{A}{j}\right)\right], \quad n\geq 2.
\end{equation}
First, we can use Stirling's formula to obtain
\begin{equation}
	\prod_{j=2}^{n-1} \left(1+\frac{A}{j}\right) = \frac{\Gamma(n+A)}{\Gamma(n)\Gamma(2+A)} \sim \frac{n^A}{\Gamma(2+A)},
\end{equation}
as $n\rightarrow\infty$. Similarly,
\begin{equation}
\begin{split}
	\sum_{i=2}^{n-1} \frac{1}{i^{2-\alpha}} \prod_{j=i+1}^{n-1} \left(1+\frac{A}{j}\right) &= \frac{\Gamma(n+A)}{\Gamma(n)}\sum_{i=2}^{n-1} \frac{1}{i^{2-\alpha}} \frac{\Gamma(i+1)}{\Gamma(i+1+A)} \\
	&\sim n^A \left(\frac{1}{\Gamma(2+A)}+\sum_{i=2}^{\infty} \frac{1}{i^{2-\alpha}} \frac{\Gamma(i+1)}{\Gamma(i+1+A)} \right),
\end{split}
\end{equation}
as $n\rightarrow\infty$. Therefore, because $nh$ remains bounded for large $n$, we have
\begin{equation}
	e_n \leq \text{const.} \times \delta(h) h^{\alpha-A-1} \quad \text{as} \quad n\rightarrow \infty.
\end{equation} 
This concludes the proof.
\end{proof}

Up to this point we have a result stating that a family of numerical methods (\ref{eqn:NumMet}) with $\delta_n(h) \leq 0$ will be convergent provided that the constant $A$ is sufficiently small. Notice that since by Lemma \ref{thm:IC} the solution $y$ is not regular at $z\rightarrow 0^+$ we do not have to (and ought to!) use a high-order quadrature. Moreover, since $y(z)^{m+1}$ is Lipschitz continuous for $X\leq z\leq 1$, the classical theory of numerical methods for integral equations works on that interval. Hence, we are only interested in solving (\ref{eqn:mainEqInt}) for the neighbourhood of zero, i.e. $0\leq z\leq X$. Eventually, $X$ can be made sufficiently small. 

As an example of the above we choose the quadrature to be the midpoint method and formulate the result as a corollary.
\begin{cor}
Assume that $0\leq z\leq X<1$ and let $y_n$ be the calculated from (\ref{eqn:NumMet}) approximation to the solution of (\ref{eqn:mainEqInt}). If the quadrature (\ref{eqn:LocCons}) is chosen to be the midpoint method, i.e. $y_1$ chosen according to (\ref{eqn:startingValue}) and
\begin{equation}
	y_{2n+k}^{m+1} = \frac{1}{2}h K_{2n+k,k} \;y_k + 2 h \sum_{i=1}^{n} K_{2n+k,2i+k-1} \; y_{2i+k-1}, \quad k\in\left\{0,1\right\}, \quad n > 1,
	\label{eqn:Midpoint}
\end{equation}
then it is convergent provided that
\begin{equation}
	m > 2-\alpha \quad \text{and} \quad A=\frac{2 D}{(m+1)C_1^m} < \alpha + \frac{2-\alpha}{m}.
\end{equation}
Moreover, the order of convergence is at least
\begin{equation}
	2-(2-\alpha)\left(1-\frac{1}{m}\right)-\frac{2 D}{(m+1)C_1^m}.
\label{eqn:orderTheo}
\end{equation}
\end{cor} 
\begin{proof}
We have to check whether the assumptions of the Theorem \ref{thm:convergence} are satisfied. First, it is an easy geometrical reasoning to ascertain that $\delta_n(h) \leq 0$ when $y$ is locally a concave function (for details see \cite{Plo17}). This is precisely the present case due to Lemma \ref{thm:IC} and the assumption that $m>2-\alpha$ (since $0<\alpha<1$). 

As was noted in \cite{Ded05} the midpoint quadrature for a nonsmooth function will, in our case, have an order of $\delta(h)\propto h^{1+(2-\alpha)/m}$ (in contrast with the second order for smooth functions). Since we are initializing the iteration from (\ref{eqn:startingValue}) which is an asymptotic form of the solution at $z\rightarrow 0^+$, the starting error will be $\mu=o(h^{(2-\alpha)/m})$. The number $(2-\alpha)/m$ is always greater than $1+(2-\alpha)/m - 2 + \alpha$ and hence $|e_1|$ will be smaller than $\text{const.} \times \delta(h) h^{\alpha-2}$ for sufficiently small $h$. The assumption of Theorem \ref{thm:convergence} concerning the initial step is thus satisfied. 

Finally, we can estimate the convergence error. From Theorem \ref{thm:convergence} we have 
\begin{equation}
	|e_n| \leq \text{const.} \times h^{1-(2-\alpha)\left(1-\frac{1}{m}\right)-A + 1}.
\end{equation}
To prove the convergence we have to show that the exponent is a positive constant. To this end notice that by the assumption
\begin{equation}
	1-(2-\alpha)\left(1-\frac{1}{m}\right)-A + 1 > 1-(2-\alpha)\left(1-\frac{1}{m}\right) - \alpha - \frac{2-\alpha}{m} + 1 = 1 - 2 + \alpha - \alpha +1 = 0.
\end{equation}
Therefore the method is convergent. 
\end{proof}

It can be verified that $C_1^m$ and $D$ are bounded for any $m$ and, hence, the last term in (\ref{eqn:orderTheo}) can be arbitrarily close to zero for sufficiently large $m$. Moreover, we can see that the theoretical estimate on the convergence error converges to $\alpha$ for large $m$. Therefore, the method is convergent for any $0<\alpha<1$ for sufficiently large $m$. 

We illustrate the above result by a series of numerical examples computing the solution of (\ref{eqn:mainEqInt}) by (\ref{eqn:Midpoint}) for different values of $\alpha$ and $m$. In each simulation we approximate the order of convergence by use of the Aitken's Method (it finds the order by extrapolation and refinement of the grid with $N$, $2N$ and $4N$ steps, see \cite{Lin85}) and compare it to the theoretical estimate (\ref{eqn:orderTheo}). The results obtained in MATLAB are given in Tab. \ref{tab:simulations}.

\begin{table}
	\centering
	\begin{tabular}{llcc}
		\toprule
		$\alpha$ & $m$ & theoretical order (\ref{eqn:orderTheo}) & empirical order \\
		\midrule
		\multirow{2}{*}{$\alpha=0.1$} & $m=1463$ & 0 & 0.83 \\
		& $m=10000$ & 0.09 & 0.98  \\
		\midrule
		\multirow{2}{*}{$\alpha=0.2$} & $m=252$ & 0 & 0.64 \\
		& $m=1000$ & 0.15 & 0.94 \\
		\midrule
		\multirow{2}{*}{$\alpha=0.3$} & $m=80$ & 0 &  0.58 \\
		& $m=100$ & 0.06 & 0.63 \\
		\midrule
		\multirow{2}{*}{$\alpha=0.4$} & $m=33$ & 0 & 0.60 \\
		& $m=100$ & 0.27 & 0.79 \\
		\midrule
		\multirow{2}{*}{$\alpha=0.5$} & $m=15$ & 0 & 0.66 \\
		& $m=100$ & 0.42 & 0.88 \\
		\midrule
		\multirow{2}{*}{$\alpha=0.6$} & $m=1$ & 0.16 & 1.01 \\
		& $m=10$ & 0.55 & 0.93 \\
		\midrule
		\multirow{2}{*}{$\alpha=0.7$} & $m=1$ & 0.44 & 0.90 \\
		& $m=10$ & 0.66 & 0.95 \\
		\midrule
		\multirow{2}{*}{$\alpha=0.8$} & $m=1$ & 0.68 & 0.87 \\
		& $m=10$ & 0.78 & 0.97 \\
		\midrule
		\multirow{2}{*}{$\alpha=0.9$} & $m=1$ & 0.88  & 0.85 \\
		& $m=10$ & 0.77 & 0.97 \\
		\midrule
		\multirow{2}{*}{$\alpha=0.99$} & $m=1$ & 1.04 & 0.83 \\
		& $m=10$ & 1.08 & 1.00 \\
		\bottomrule
	\end{tabular}
	\caption{Results of the simulations done for $N=3000$. For each $\alpha$ the corresponding $m$ has been chosen for the theoretical estimate on the order (\ref{eqn:orderTheo}) to be non-negative. }
	\label{tab:simulations}
\end{table}

First, we have to remark that obtaining an accurate value of the order of convergence is very demanding on the computer power. We have settled to choosing $N=3\cdot 10^3$ as the number of subdivisions of the $[0,1]$ interval (hence, the maximal division in Aitken's method is $12\cdot 10^3$). 

Notice that for $\alpha\rightarrow 1^-$ the method order approaches $1$ both theoretically (for large $m$) and empirically. For smaller values of $\alpha$ the empirical order gets lower but this fact can be due to the singularity of the kernel (\ref{eqn:kernelExp}) at $\alpha = 0$, which is difficult to resolve numerically. In almost every case the theoretical estimate is lower than the numerically found value of the order. This was anticipated since the value (\ref{eqn:orderTheo}) is not optimal - it depends on the bound from Lemma \ref{lem:estimates} which might not be accurate enough. The three small exceptions to the above, that is $(\alpha,m)=(0.9,1)$ and $(0.99,1)$ and $(0.99,10)$, are most probably caused by not sufficiently large $N$ used in our simulations to resolve the empirical order. It also may be conjectured that, guessing from our simulations, the true order of the method is equal to $1$ for, at least, $\alpha$ close to unity. The verification of this claim requires more refined proof techniques and numerical simulations which is the main goal of our future work. 

\section{Conclusion}
We have devised a convergent numerical method for solving the nonlocal nonlinear porous medium equation. Due to the non-Lipschitzian nonlinearity, the classical proof methods could have not been used. Our approach gives an estimate on the convergence error but it is clear that it does not cover all of the admissible $(\alpha,m)$ parameter space and is not optimal. One of the reasons is the $C_1$ constant which enlarge the essential exponent $A$. The object of our future work will be to overcome this difficulty and obtain more strict bounds on the convergence order.

\section*{Acknowledgement}
This research was supported by the National Science Centre, Poland under the project with a signature NCN $2015/17/D/ST1/00625$.


\end{document}